\documentclass[preprint,12pt,3p]{elsarticle}

\usepackage{amssymb}
\usepackage{amsmath}
\usepackage{amsfonts}
\usepackage{amsthm}
\usepackage{graphics,graphicx}
\usepackage{float}
\usepackage{mathtools}

\usepackage[english]{babel}

\journal{Arxiv}
\newtheorem{theorem}{Theorem}
\newtheorem{lemma}{Lemma}

\newtheorem{corollary}{Corollary}
\newtheorem{example}{Example}

\newtheorem{remark}{Remark}
\newtheorem{problem}{Problem}
\begin{document}

\begin{frontmatter}

\title{Eigenfunctions and minimum 1-perfect bitrades in the Hamming graph \tnoteref{label0}}
\tnotetext[label0]{This work was funded by the Russian Science Foundation under grant 18-11-00136.}

\author[01]{Alexandr~Valyuzhenich}
\ead{graphkiper@mail.ru}

\address[01]{Sobolev Institute of Mathematics, Ak. Koptyug av. 4, Novosibirsk 630090, Russia}

\begin{abstract}
The Hamming graph $H(n,q)$ is the graph whose vertices are the words of length $n$ over the alphabet $\{0,1,\ldots,q-1\}$, where two vertices are adjacent if they differ in exactly one coordinate. The adjacency matrix of $H(n,q)$ has $n+1$ distinct eigenvalues
$n(q-1)-q\cdot i$  with corresponding eigenspaces $U_{i}(n,q)$ for $0\leq i\leq n$.
In this work we study functions belonging to a direct sum $U_i(n,q)\oplus U_{i+1}(n,q)\oplus\ldots\oplus U_j(n,q)$ for $0\leq i\leq j\leq n$.
We find the minimum cardinality of the support of such functions for $q=2$ and for $q=3$, $i+j>n$.
In particular, we find the minimum cardinality of the support of eigenfunctions from the eigenspace $U_{i}(n,3)$ for $i>\frac{n}{2}$.
Using the correspondence between $1$-perfect bitrades and eigenfunctions with eigenvalue $-1$,
we find the minimum size of a $1$-perfect bitrade in the Hamming graph $H(n,3)$.
\end{abstract}

\begin{keyword}
Hamming graph\sep eigenfunction\sep eigenfunctions of graphs\sep eigenspace\sep minimum support\sep trade\sep bitrade\sep 1-perfect bitrade
\vspace{\baselineskip}
\MSC[2010] 05C50\sep 05B30\sep 05E30
\end{keyword}

\end{frontmatter}


\section{Introduction}
In this work we consider the following extremal problem for eigenfunctions of graphs.
\begin{problem}\label{ProblemMinSupport}
Let $G$ be a graph and let $\lambda$ be an eigenvalue of the adjacency matrix of $G$.
Find the minimum cardinality of the support of a $\lambda$-eigenfunction of $G$.
\end{problem}
Problem~\ref{ProblemMinSupport} is directly related to the intersection problem of
two combinatorial objects and to the problem of finding the minimum cardinality of bitrades.
Often such problems can be considered as Problem \ref{ProblemMinSupport} for the corresponding graph and some eigenvalue with some additional discrete restrictions on the functions. In this context we would like to mention the papers by Graham et al. \cite{Graham}, Deza and Frankl \cite{DezaFrankl}, Frankl and Pach \cite{FranklPach} and Cho \cite{ChoLAA,ChoCombinatorica} on null designs and the paper by Hwang \cite{Hwang} on combinatorial trades. In more details, connections between eigenfunctions and bitrades are described in \cite{Krotovtezic,KrotovPerfectBitrades,KrotovMogilnykhPotapov,VorobevKrotov}.

Problem~\ref{ProblemMinSupport} was studied for the bilinear forms graphs in~\cite{SotnikovaBilinear}, for the cubical distance-regular graphs in~\cite{SotnikovaCubical}, for the Doob graphs in~\cite{Bespalov}, for the Grassmann graphs in \cite{ChoLAA,ChoCombinatorica,KrotovMogilnykhPotapov}, for the Hamming graphs in~\cite{Krotovtezic,Potapov,Valyuzhenich,VorobevKrotov}, for the Johnson graphs in~\cite{VMV}, for the Paley graphs in~\cite{GoryainovKabanovShalaginovValyuzhenich} and for the Star graph in \cite{KabanovKonstantinovaShalaginovValyuzhenich}.

The Hamming graph $H(n,q)$ is a graph whose vertices are the words of length $n$ over the alphabet $\{0,1,\ldots,q-1\}$; and two vertices are adjacent if they differ in exactly one coordinate.
The adjacency matrix of $H(n,q)$ has $n+1$ eigenvalues
$\lambda_{i}(n,q)=n(q-1)-q\cdot i$, where $0\leq i\leq n$.
Let $U_{[i,j]}(n,q)$, where $0\leq i\leq j\leq n$, denote a direct sum of eigenspaces of $H(n,q)$ corresponding to consecutive eigenvalues from $n(q-1)-q\cdot i$ to $n(q-1)-q\cdot j$. In this work we consider the following generalization of Problem \ref{ProblemMinSupport} for the Hamming graph.

\begin{problem}\label{ProblemHamming}
Let $n\geq 1$, $q\ge 2$ and $0\leq i\leq j\leq n$. Find the minimum cardinality of the support of functions from the space $U_{[i,j]}(n,q)$.
\end{problem}
In \cite{ValyuzhenichVorobev} Valyuzhenich and Vorob'ev solved Problem \ref{ProblemHamming} for arbitrary $q\geq 3$ except the case when $q=3$ and $i+j>n$.
Moreover, in \cite{ValyuzhenichVorobev} a characterization of functions from the space $U_{[i,j]}(n,q)$ with the minimum cardinality of the support
was obtained for $q\ge 3$, $i+j\le n$ and $q\ge 5$, $i=j$, $i>\frac{n}{2}$.
In this work we solve Problem \ref{ProblemHamming} for $q=2$ and $q=3$, $i+j>n$.
In particular, we find the minimum cardinality of the support of a $\lambda_{i}(n,3)$-eigenfunction of $H(n,3)$ for $i>\frac{n}{2}$.
Thus, Problem \ref{ProblemHamming} is now completely solved.
As we see below, in the case $q=3$ the eigenfunctions attaining the minimum cardinality of the support have more complicated structure than for $q>3$, so the case remaining after the preceding work and solved in the current paper is really exceptional.

Bitrades are used for constructing and studying combinatorial designs and codes (see \cite{Billington,Cavenagh,HedayatKhosrovshahi}).
One of important problems in the theory of bitrades is the problem of finding the minimum sizes of bitrades.
This problem was investigated for null designs \cite{ChoLAA,ChoCombinatorica,FranklPach}, for combinatorial bitrades \cite{Hwang}, for Latin bitrades \cite{PotapovLatin} and for q-ary Steiner bitrades \cite{KrotovTheminimumvolume,KrotovMogilnykhPotapov}.
In this work we study $1$-perfect bitrades in the Hamming graph.
The problems of the existence and classification of $1$-perfect bitrades and extended $1$-perfect bitrades in the Hamming graphs were studied in \cite{MogilnykhSolov'eva,VorobevKrotov} and in \cite{KrotovPerfectBitrades} respectively.
In this work we consider the following problem for $1$-perfect bitrades in the Hamming graph.
\begin{problem}\label{ProblemPerfectBitrade}
Let $n\ge 3$ and $q\ge 2$. Find the minimum size of a $1$-perfect bitrade in $H(n,q)$.
\end{problem}
For $q=2$ Problem \ref{ProblemPerfectBitrade} was essentially solved by Etzion and Vardy \cite{EtzionVardy} and Solov'eva \cite{Solov'eva}
(the results were formulated for more special cases of $1$-perfect bitrades embedded into perfect binary codes, but both proofs work in the general case).
In \cite{MogilnykhSolov'eva} Mogilnykh and Solov'eva for arbitrary $q\geq 2$  showed the existence of $1$-perfect bitrades in $H(n,q)$ of size
$2\cdot (q!)^\frac{n-1}{q}$. This fact implies that a lower bound $2^{n-\frac{n-1}{q}}\cdot (q-1)^\frac{n-1}{q}$ for the size of $1$-perfect bitrades in $H(n,q)$ proved in \cite{ValyuzhenichVorobev} is sharp for $q=4$, i.e. Problem \ref{ProblemPerfectBitrade} is solved for $q=4$.
In \cite{MogilnykhSolov'eva} Mogilnykh and Solov'eva found the minimum size of a $1$-perfect bitrade in $H(q+1,q)$ for arbitrary $q\geq 2$.
In this work, using the correspondence between $1$-perfect bitrades and $(-1)$-eigenfunctions, we solve Problem \ref{ProblemPerfectBitrade} for $q=3$.

The paper is organized as follows. In Section \ref{SectionDefinitions}, we introduce basic definitions and notations.
In Section \ref{SectionPreliminaries}, we give some preliminary results.
In Section \ref{SectionConstructionsFunctions}, we define four families of functions that have the minimum cardinality of the support in the space $U_{[i,j]}(n,q)$ for $q=2$ and for $q=3$ and $i+j>n$ respectively.
In Section \ref{SectionProblem2 for q=2}, we find the minimum cardinality of the support of functions from the space  $U_{[i,j]}(n,2)$.
In Section \ref{SectionProblem2 for q=3,i/2+j<=n}, we find the minimum cardinality of the support of functions from the space $U_{[i,j]}(n,3)$ for $\frac{i}{2}+j\leq n$ and $i+j>n$.
In Section \ref{SectionProblem2 for q=3,i/2+j>n}, we find the minimum cardinality of the support of functions from the space  $U_{[i,j]}(n,3)$ for $\frac{i}{2}+j>n$.
In Section \ref{SectionBitrades}, we prove that the minimum size of a $1$-perfect bitrade in $H(3m+1,3)$, where $m\ge 1$, is $2^{m+1}\cdot 3^m$.
\section{Basic definitions}\label{SectionDefinitions}
Let $G=(V,E)$ be a graph with the adjacency matrix $A(G)$. The set of neighbors of a vertex $x$ is denoted by $N(x)$.
Let $\lambda$ be an eigenvalue of the matrix $A(G)$. A function $f:V\longrightarrow{\mathbb{R}}$ is called a {\em $\lambda$-eigenfunction} of $G$ if $f\not\equiv 0$ and the equality
\begin{equation}\label{EigenfunctionDefinition}
\lambda\cdot f(x)=\sum_{y\in{N(x)}}f(y)
\end{equation}
holds for any vertex $x\in V$.
The set of functions $f:V\longrightarrow{\mathbb{R}}$ satisfying the equality (\ref{EigenfunctionDefinition}) for any vertex $x\in V$ is called a {\em $\lambda$-eigenspace} of $G$.
The {\em support} of a function $f:V\longrightarrow{\mathbb{R}}$ is the set $Supp(f)=\{x\in V~|~f(x)\neq 0\}$.
Denote $|f|=|Supp(f)|$.

Let $\Sigma_q=\{0,1,\ldots,q-1\}$.
The vertex set of the {\em Hamming graph} $H(n,q)$ is $\Sigma_{q}^n$ and two vertices are adjacent if they differ in exactly one coordinate.
It is well known that the set of eigenvalues of the adjacency matrix of $H(n,q)$ is $\{\lambda_{i}(n,q)=n(q-1)-q\cdot i\mid i=0,1,\ldots,n\}$.
Denote by $U_{i}(n,q)$ the $\lambda_{i}(n,q)$-eigenspace of $H(n,q)$. The direct sum of subspaces
$$U_i(n,q)\oplus U_{i+1}(n,q)\oplus\ldots\oplus U_j(n,q)$$ for $0\leq i\leq j\leq n$ is denoted by $U_{[i,j]}(n,q)$.

The {\em Cartesian product} $G\square H$ of graphs $G$ and $H$ is a graph with the vertex set $V(G)\times V(H)$; and
any two vertices $(u,u')$ and $(v,v')$ are adjacent if and only if either
$u=v$ and $u'$ is adjacent to $v'$ in $H$, or
$u'=v'$ and $u$ is adjacent to $v$ in $G$.

Let $G_1=(V_1,E_1)$ and $G_2=(V_2,E_2)$ be two graphs, let $f_1:V_1\longrightarrow{\mathbb{R}}$ and $f_2:V_2\longrightarrow{\mathbb{R}}$.
Let $G=G_1\square G_2$.
Define the {\em tensor product} $f_1\cdot f_2:V_1\times V_2\longrightarrow\mathbb{R}$ by the following rule: $(f_1\cdot f_2)(x,y)=f_1(x)f_2(y)$ for $(x,y)\in V(G)=V_1\times V_2$.

Let $y=(y_1,\ldots,y_{n-1})$ be a vertex of $H(n-1,q)$, $k\in{\Sigma_q}$ and $r\in\{1,2,\ldots,n\}$. We consider the vector $x=(y_1,\ldots,y_{r-1},k,y_{r},\ldots,y_{n-1})$ of length $n$.
Given a function $f:\Sigma_{q}^{n}\longrightarrow{\mathbb{R}}$,  we define the function $f_{k}^{r}:\Sigma_{q}^{n-1}\longrightarrow{\mathbb{R}}$ by the rule $f_{k}^{r}(y)=f(x)$. A function $f:\Sigma_{q}^n\longrightarrow{\mathbb{R}}$ is called {\em uniform} if for any $r\in\{1,2,\ldots,n\}$ there exists $l(r)\in\Sigma_q$ such that $f_{k}^{r}=f_{m}^{r}$ for all $k,m\in{\Sigma_q\setminus \{l(r)\}}$.

Let $\rm{Sym}(X)$ denote the symmetric group on a finite set $X$ and let $\rm{Sym}_{n}$ denote the symmetric group on the set $\{1,\ldots,n\}$.

Let $f(x_1,x_2,\ldots,x_n)$ be a function, let $\pi\in \rm{Sym}_n$ and let $\sigma_1,\ldots,\sigma_{n}\in \rm{Sym}(\Sigma_q)$.  We define the function $f_{\pi,\sigma_{1},\ldots,\sigma_{n}}$ by the following rule:
$$f_{\pi,\sigma_{1},\ldots,\sigma_{n}}(x_1,\ldots,x_n)=f(\sigma_{1}(x_{\pi(1)}),\ldots,\sigma_{n}(x_{\pi(n)})).$$

Let $G=(V,E)$ be a graph. For a vertex $x\in V$ denote $B(x)=N(x)\cup \{x\}$.
Let $T_0$ and $T_1$ be two disjoint nonempty subsets of $V$.
The ordered pair $(T_0,T_1)$ is called a {\em $1$-perfect bitrade} in $G$ if for any vertex $x\in V$ the set $B(x)$ either contains one element from $T_0$ and one element from $T_1$ or does not contain elements from $T_0\cup T_1$.

\begin{remark}\label{RemarkBitrade}
If $(T_0,T_1)$ is a $1$-perfect bitrade in a graph $G$, then the following properties hold:
\begin{itemize}
  \item $T_0$ and $T_1$ are independent sets in $G$ and $|T_0|=|T_1|$.
  \item the subgraph of $G$ induced by $T_0\cup T_1$ is a perfect matching.
\end{itemize}
\end{remark}

The {\em size} of a $1$-perfect bitrade $(T_0,T_1)$ is $|T_0|+|T_1|$.

\begin{example}\label{ExamplePerfectBitrade}
Let $T_0=\{000,111\}$ and $T_1=\{001,110\}$. Then $(T_0,T_1)$ is a $1$-perfect bitrade of size $4$ in $H(3,2)$ (see Figure \ref{Figure}).
\end{example}
\begin{figure}[H]
\begin{center}
\includegraphics[scale=0.4]{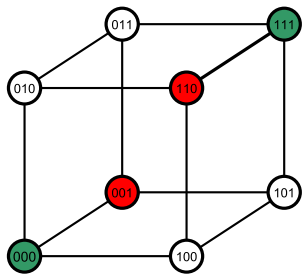}
\end{center}
\caption{$1$-perfect bitrade in $H(3,2)$.}\label{Figure}
\end{figure}

\begin{example}\label{ExamplePerfectBitradeCodes}
Let $G=(V,E)$ be a graph. Recall that a set $C\subseteq V$ is called a {\em $1$-perfect code} in $G$ if for any vertex $x\in V$ the set $B(x)$ contains one vertex from $C$. Let $C_1$ and $C_2$ be two $1$-perfect codes in $G$ ($C_1\neq C_2$). Then $(C_1\setminus C_2,C_2\setminus C_1)$ is a $1$-perfect bitrade in $G$.
\end{example}

Let $(T_0,T_1)$ be a $1$-perfect bitrade in a graph $G=(V,E)$. We define the function $f_{(T_0,T_1)}:V\longrightarrow{\{-1,0,1\}}$ by the following rule:
$$
f_{(T_0,T_1)}(x)=\begin{cases}
1,&\text{if $x\in T_0$;}\\
-1,&\text{if $x\in T_1$;}\\
0,&\text{otherwise.}
\end{cases}
$$

\section{Preliminaries}\label{SectionPreliminaries}
In this section we give useful preliminary results. The following result is a corollary of well known result for so-called NEPS of graphs (see \cite{Cvetkovic}, Theorem 2.3.4).
\begin{lemma}[\cite{ValyuzhenichVorobev}, Corollary 1]\label{product of functions}
Let $f_1\in{U_{i}(m,q)}$ and $f_2\in{U_{j}(n,q)}$. Then $f_1\cdot f_2\in{U_{i+j}(m+n,q)}$.
\end{lemma}
We will use Lemma \ref{product of functions} in Section \ref{SectionConstructionsFunctions}.
The following two results were proved in \cite{ValyuzhenichVorobev}.
\begin{lemma}[\cite{ValyuzhenichVorobev}, Lemma 4]\label{Reduction 1}
Let $f\in{U_{[i,j]}(n,q)}$ and $r\in\{1,2,\ldots,n\}$. Then the following statements are true:
\begin{enumerate}
\item $f_{k}^{r}-f_{m}^{r}\in{U_{[i-1,j-1]}(n-1,q)}$ for $k,m\in\Sigma_q$.
\item $\sum_{k=0}^{q-1}f_{k}^{r}\in{U_{[i,j]}(n-1,q)}$.
\item $f_{k}^{r}\in{U_{[i-1,j]}(n-1,q)}$ for $k\in\Sigma_q$.
\end{enumerate}
\end{lemma}

\begin{lemma}[\cite{ValyuzhenichVorobev}, Lemma 5]\label{Reduction 2}
Let $f\in{U_{[i,j]}(n,q)}$, let $r\in\{1,2,\ldots,n\}$, and let $m\in\Sigma_q$. If $f_{k}^{r}\equiv 0$ for any $k\in\Sigma_{q}\setminus\{m\}$, then $f_{m}^{r}\in{U_{[i,j-1]}(n-1,q)}$.
\end{lemma}
In Sections \ref{SectionProblem2 for q=2}, \ref{SectionProblem2 for q=3,i/2+j<=n} and \ref{SectionProblem2 for q=3,i/2+j>n} we will use Lemmas \ref{Reduction 1} and \ref{Reduction 2} for inductive arguments.
The following two results were proved in \cite{ValyuzhenichVorobev}.
\begin{theorem}[\cite{ValyuzhenichVorobev}, Theorem 2]\label{TheoremPrelim for uniform}
Let $f$ be a uniform function from $U_{[i,j]}(n,q)$, where $i+j\geq n$, $q\geq 3$ and $f\not\equiv 0$. Then $|f|\geq{2^{n-j}(q-1)^{n-j}q^{i+j-n}}$.
\end{theorem}

\begin{theorem}[\cite{ValyuzhenichVorobev}, Theorem 1]\label{TheoremPrelim q=3,i+j<=n}
Let $f\in{U_{[i,j]}(n,q)}$, $i+j\le n$, $q\geq 3$ and $f\not\equiv 0$. Then $|f|\geq 2^{i}(q-1)^{i}q^{n-i-j}$.
\end{theorem}
We will use Theorems \ref{TheoremPrelim for uniform} and \ref{TheoremPrelim q=3,i+j<=n} in the proof of Theorem \ref{Theorem 5}.
The following result was obtained in \cite{VorobevKrotov}.
\begin{lemma}[\cite{VorobevKrotov}, Proposition 2]\label{BitradeConstruction}
Let $n=qm+1$ and $q=p^k$, where $p$ is a prime, $m\ge 1$ and $k\ge 1$. Then there exist a $1$-perfect bitrade in $H(n,q)$ of size $2^{m+1}\cdot q^{m(q-2)}$.
\end{lemma}
We will use Lemma \ref{BitradeConstruction} in the proof of Theorem \ref{TheoremBitrade}.
\section{Constructions of functions with the minimum cardinality of the support}\label{SectionConstructionsFunctions}
In this section we give constructions of functions that have the minimum cardinality of the support in the space $U_{[i,j]}(n,q)$ for $q=2$ and $q=3$, $i+j>n$.

We define the function $a_{q,k,m}:\Sigma_{q}^2\longrightarrow{\mathbb{R}}$ for $k,m\in{\Sigma_q}$ by the following rule:
$$
a_{q,k,m}(x,y)=\begin{cases}
1,&\text{if $x=k$ and $y\neq m$;}\\
-1,&\text{if $y=m$ and $x\neq k$;}\\
0,&\text{otherwise.}
\end{cases}
$$
We note that $|a_{q,k,m}|=2(q-1)$ and $a_{q,k,m}\in U_{1}(2,q)$ for any $k,m\in{\Sigma_q}$. Denote $A_q=\{a_{q,k,m}~|~k,m\in{\Sigma_q}\}$.

We define the function $\varphi_1:\Sigma_{3}^2\longrightarrow{\mathbb{R}}$ by the following rule:
$$
\varphi_1(x,y)=\begin{cases}
1,&\text{if $x=y=0$;}\\
-1,&\text{if $x=1$ and $y=2$;}\\
0,&\text{otherwise.}
\end{cases}
$$

For $a,b\in{\Sigma_{3}}$ denote by $a\oplus b$ the sum of $a$ and $b$ modulo $3$.
We define the function $\varphi:\Sigma_{3}^3\longrightarrow{\mathbb{R}}$ by the following rule:
$$
\varphi(x,y,z)=\begin{cases}
\varphi_1(x,y),&\text{if $z=0$;}\\
\varphi_1(x\oplus1,y\oplus1),&\text{if $z=1$;}\\
\varphi_1(x\oplus2,y\oplus2),&\text{if $z=2$.}
\end{cases}
$$
We note that $|\varphi|=6$. By the definition of an eigenfunction we see that $\varphi\in{U_2(3,3)}$.
Denote $$B=\{\varphi_{\pi,\sigma_{1},\sigma_{2},\sigma_{3}}~|~\pi\in \rm{Sym}_3,\sigma_{1},\sigma_{2},\sigma_{3}\in \rm{Sym}(\Sigma_3)\}.$$

We define the function $c_{q,k,m}:\Sigma_{q}\longrightarrow{\mathbb{R}}$ for $k,m\in{\Sigma_q}$ and $k\neq m$ by the following rule:
$$
c_{q,k,m}(x)=\begin{cases}
1,&\text{if $x=k$;}\\
-1,&\text{if $x=m$;}\\
0,&\text{otherwise.}
\end{cases}
$$
We note that $|c_{q,k,m}|=2$ and $c_{q,k,m}\in U_{1}(1,q)$ for any $k,m\in{\Sigma_q}$ and $k\neq m$.
Denote $C_q=\{c_{q,k,m}~|~k,m\in{\Sigma_q},k\neq m\}$.

We define the function $d_{q,k}:\Sigma_{q}\longrightarrow{\mathbb{R}}$ for $k\in{\Sigma_q}$ by the following rule:
$$
d_{q,k}(x)=\begin{cases}
1,&\text{if $x=k$;}\\
0,&\text{otherwise.}
\end{cases}
$$
We note that $|d_{q,k}|=1$ and $d_{q,k}\in U_{[0,1]}(1,q)$ for any $k\in{\Sigma_q}$.
Denote $D_q=\{d_{q,k}~|~k\in{\Sigma_q}\}$.

Let $e_q:\Sigma_{q}\longrightarrow{\mathbb{R}}$ and $e_q\equiv 1$.
We note that $|e_{q}|=q$ and $e_{q}\in U_{0}(1,q)$.
Denote $E_q=\{e_q\}$.

\begin{figure}[H]
\begin{center}
\includegraphics[scale=0.37]{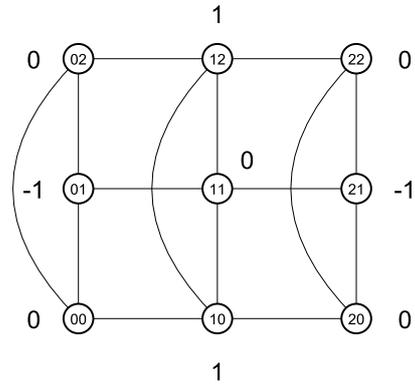}
\end{center}
\caption{Function $a_{3,1,1}$ in $H(2,3)$.}\label{Figure1}
\end{figure}

\begin{figure}[H]
\begin{center}
\includegraphics[scale=0.3]{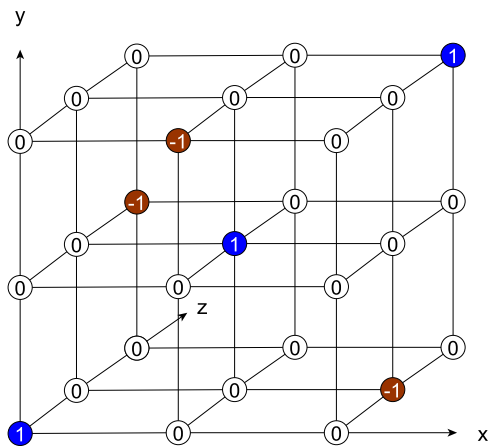}
\end{center}
\caption{Function $\varphi(x,y,z)$ in $H(3,3)$.}\label{Figure2}
\end{figure}

\begin{figure}[H]
\begin{center}
\includegraphics[scale=0.35]{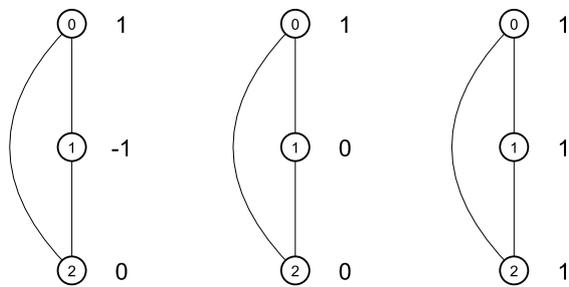}
\end{center}
\caption{Functions $c_{3,0,1}$, $d_{3,0}$ and $e_3$ in $H(1,3)$.}\label{Figure3}
\end{figure}

Let $i+j\le n$. We say that a function $f:\Sigma_{2}^n\longrightarrow{\mathbb{R}}$ belongs to the class $F_{1}(n,i,j)$ if $$f=c\cdot \prod_{k=1}^{i}g_{k}\cdot \prod_{k=1}^{n-i-j}h_{k}\cdot \prod_{k=1}^{j-i}v_{k},$$ where $c$ is a constant, $g_k\in{A_2}$ for $k\in{[1,i]}$, $h_k\in{E_2}$ for $k\in{[1,n-i-j]}$ and $v_k\in{D_2}$ for $k\in{[1,j-i]}$.

Let $i+j>n$. We say that a function $f:\Sigma_{2}^n\longrightarrow{\mathbb{R}}$ belongs to the class $F_{2}(n,i,j)$ if $$f=c\cdot \prod_{k=1}^{n-j}g_{k}\cdot \prod_{k=1}^{i+j-n}h_{k}\cdot \prod_{k=1}^{j-i}v_{k},$$ where $c$ is a constant, $g_k\in{A_2}$ for $k\in{[1,n-j]}$, $h_k\in{C_2}$ for $k\in{[1,i+j-n]}$ and $v_k\in{D_2}$ for $k\in{[1,j-i]}$.

Let $\frac{i}{2}+j\leq n$ and $i+j>n$. We say that a function $f:\Sigma_{3}^n\longrightarrow{\mathbb{R}}$ belongs to the class $F_{3}(n,i,j)$ if $$f=c\cdot \prod_{k=1}^{2n-i-2j}g_{k}\cdot \prod_{k=1}^{i+j-n}h_{k}\cdot \prod_{k=1}^{j-i}v_{k},$$ where $c$ is a constant, $g_k\in{A_3}$ for $k\in{[1,2n-i-2j]}$, $h_k\in{B}$ for $k\in{[1,i+j-n]}$ and $v_k\in{D_3}$ for $k\in{[1,j-i]}$.

Let $\frac{i}{2}+j>n$. We say that a function $f:\Sigma_{3}^n\longrightarrow{\mathbb{R}}$ belongs to the class $F_{4}(n,i,j)$ if $$f=c\cdot \prod_{k=1}^{n-j}g_{k}\cdot \prod_{k=1}^{i+2j-2n}h_{k}\cdot \prod_{k=1}^{j-i}v_{k},$$ where $c$ is a constant, $g_k\in{B}$ for $k\in{[1,n-j]}$, $h_k\in{C_3}$ for $k\in{[1,i+2j-2n]}$ and $v_k\in{D_3}$ for $k\in{[1,j-i]}$.
\begin{lemma}\label{constructions of eigenfunctions}
The following statements are true:
\begin{enumerate}
\item Let $i+j\le n$ and $f\in{F_{1}(n,i,j)}$. Then $f\in{U_{[i,j]}(n,2)}$ and $|f|=2^{n-j}$.

  \item Let $i+j>n$ and $f\in{F_{2}(n,i,j)}$. Then $f\in{U_{[i,j]}(n,2)}$ and $|f|=2^{i}$.

\item Let $\frac{i}{2}+j\leq n$, $i+j>n$ and $f\in{F_{3}(n,i,j)}$. Then $f\in{U_{[i,j]}(n,3)}$ and $|f|=2^{3(n-j)-i}\cdot 3^{i+j-n}$.

  \item Let $\frac{i}{2}+j>n$ and $f\in{F_{4}(n,i,j)}$. Then $f\in{U_{[i,j]}(n,3)}$ and $|f|=2^{i+j-n}\cdot 3^{n-j}$.

\end{enumerate}

\end{lemma}
\begin{proof}
As we noted above $A_q\subset U_{1}(2,q)$, $B\subset U_{2}(3,3)$, $C_q\subset U_{1}(1,q)$, $D_q\subset U_{[0,1]}(1,q)$ and $E_q\subset U_{0}(1,q)$. Hence using Lemma~\ref{product of functions} and the fact that $|f_1\cdot f_2|=|f_1|\cdot |f_2|$, we obtain the statement of this lemma.
\end{proof}

In Section \ref{SectionProblem2 for q=2} we prove that functions from $F_{1}(n,i,j)$ and $F_{2}(n,i,j)$ have the minimum cardinality of the support in the space $U_{[i,j]}(n,2)$ for $i+j\leq n$ and $i+j>n$ respectively. In Sections \ref{SectionProblem2 for q=3,i/2+j<=n} and \ref{SectionProblem2 for q=3,i/2+j>n} we prove that functions from $F_{3}(n,i,j)$ and $F_{4}(n,i,j)$ have the minimum cardinality of the support in the space $U_{[i,j]}(n,3)$ for
$\frac{i}{2}+j\leq n$, $i+j>n$ and $\frac{i}{2}+j>n$ respectively.
\section{Problem \ref{ProblemHamming} for $q=2$}\label{SectionProblem2 for q=2}
In this section we consider Problem \ref{ProblemHamming} for $q=2$. The first main result of this section is the following.
\begin{theorem}\label{Theorem 3}
Let $f\in{U_{[i,j]}(n,2)}$, where $i+j\leq n$ and $f\not\equiv 0$. Then
\begin{equation}\label{Bound 2}
|f|\geq 2^{n-j}
\end{equation}
 and this bound is sharp.
\end{theorem}
\begin{proof}
Let us prove the bound (\ref{Bound 2}) by induction on $n$, $i$ and $j$.
If $j=0$, then $i=0$ and $f\in{U_{0}(n,2)}$. So, in this case $f$ is a constant. Hence $|f|=2^{n}$ and the claim of the theorem holds.
So, we can assume that $j\geq 1$. If $n=1$ and $j\geq 1$, then $i=0$ and $j=1$. In this case $f\in{U_{[0,1]}(1,2)}$ and the claim of the theorem holds.

Let us prove the induction step for $n\geq 2$ and $j\geq 1$.
Since $j\geq 1$, $f$ is not constant. Then there exists $r\in\{1,2,\ldots,n\}$ such that $f_{0}^{r}\neq f_{1}^{r}$.
Without loss of generality, we assume that $r=n$. Denote $f_k=f_{k}^n$ for $k\in{\Sigma_2}$.
Lemma \ref{Reduction 1} implies that $f_{0}-f_{1}\in{U_{[i-1,j-1]}(n-1,2)}$.
By the induction assumption we obtain that $$|f_0-f_1|\geq 2^{n-j}.$$
Then we have $$|f|=|f_0|+|f_1|\geq |f_0-f_1|\geq 2^{n-j}.$$

Lemma \ref{constructions of eigenfunctions} implies that if $f\in{F_{1}(n,i,j)}$, then $f\in{U_{[i,j]}(n,2)}$ and $|f|=2^{n-j}$.
Thus, the bound (\ref{Bound 2}) is sharp.
\end{proof}

The second main result of this section is the following.
\begin{theorem}\label{Theorem 4}
Let $f\in{U_{[i,j]}(n,2)}$, where $i+j>n$ and $f\not\equiv 0$. Then
\begin{equation}\label{Bound 3}
|f|\geq 2^{i}
\end{equation}
 and this bound is sharp.
\end{theorem}
\begin{proof}
Since $i+j>n$, we have $i\geq 1$.
Let us prove the bound (\ref{Bound 3}) by induction on $n$, $i$ and $j$.
If $n=1$ and $i+j>n$, then $i=j=1$. In this case $f\in{U_{1}(1,2)}$ and the bound $|f|\geq 2$ holds.
Let us prove the induction step for $n\geq 2$. Denote $f_k=f_{k}^n$ for $k\in{\Sigma_2}$.
Let us consider two cases.

Suppose that $f_0\not\equiv 0$ and $f_1\not\equiv 0$. Lemma \ref{Reduction 1} implies that $f_{k}\in{U_{[i-1,j]}(n-1,2)}$
for any $k\in{\Sigma_2}$. By the induction assumption we obtain that $|f_k|\geq 2^{i-1}$ for any $k\in{\Sigma_2}$.
Then we have $$|f|=|f_0|+|f_1|\geq 2^{i}.$$

Suppose that $f_k\equiv 0$ for some $k\in{\Sigma_2}$. Without loss of generality, we assume that $f_0\equiv 0$ and $f_1\not\equiv 0$.
Lemma \ref{Reduction 2} implies that $f_{1}\in{U_{[i,j-1]}(n-1,2)}$. By the induction assumption we obtain that $|f_1|\geq 2^{i}$.
Then we have $$|f|=|f_0|+|f_1|=|f_1|\geq 2^{i}.$$

Lemma \ref{constructions of eigenfunctions} implies that if $f\in{F_{2}(n,i,j)}$, then $f\in{U_{[i,j]}(n,2)}$ and $|f|=2^{i}$.
Thus, the bound (\ref{Bound 3}) is sharp.
\end{proof}
\section{Problem \ref{ProblemHamming} for $q=3$, $i+j>n$ and $\frac{i}{2}+j\leq n$}\label{SectionProblem2 for q=3,i/2+j<=n}
In this section we consider Problem \ref{ProblemHamming} for $q=3$, $i+j>n$ and $\frac{i}{2}+j\leq n$. The main result of this section is the following.
\begin{theorem}\label{Theorem 5}
Let $f\in{U_{[i,j]}(n,3)}$, where $\frac{i}{2}+j\leq n$, $i+j>n$ and $f\not\equiv 0$. Then
\begin{equation}\label{Bound 4}
|f|\geq 2^{3(n-j)-i}\cdot3^{i+j-n}
\end{equation}
 and this bound is sharp.
\end{theorem}
\begin{proof}
Let us prove the bound (\ref{Bound 4}) by induction on $n$, $i$ and $j$.
If $n\leq 3$, $i+j>n$ and $\frac{i}{2}+j\leq n$, then $n=3$ and $i=j=2$. Then $f\in U_{2}(3,3)$.
In this case the proof of the theorem can be carried out in the same way as for the induction step.

Let us prove the induction step for $n\geq 4$. If $f$ is uniform, then applying Theorem \ref{TheoremPrelim for uniform} for $q=3$ we obtain that
$$|f|\geq 2^{2(n-j)}\cdot 3^{i+j-n}>2^{3(n-j)-i}\cdot3^{i+j-n}.$$
So, we can assume that $f$ is non-uniform.
Then there exists a number $r\in\{1,\ldots,n\}$ such that $f_{k}^{r}\not\equiv f_{m}^{r}$ for any $k,m\in \Sigma_3$ and $k\neq m$.
Without loss of generality, we assume that $r=n$. Denote $f_k=f_{k}^{n}$ for $k\in \Sigma_3$.

Lemma \ref{Reduction 1} implies that $f_{k}-f_{m}\in{U_{[i-1,j-1]}(n-1,3)}$ for any $k,m\in \Sigma_3$ and $k\neq m$.
Since $\frac{i}{2}+j\leq n$ and $i+j>n$, we see that $i\geq 2$.
Moreover, we have $\frac{i-1}{2}+j-1\leq n-1$.
Then applying the induction assumption for $i+j>n+1$ and Theorem \ref{TheoremPrelim q=3,i+j<=n} for $i+j=n+1$, we obtain  that
$$|f_k-f_m|\geq 2^{3(n-j)-i+1}\cdot3^{i+j-n-1}$$
for any $k,m\in \Sigma_3$ and $k\neq m$.
Hence $$|f_k|+|f_m|\geq 2^{3(n-j)-i+1}\cdot3^{i+j-n-1}$$
for any $k,m\in \Sigma_3$ and $k\neq m$.
Then we have $$|f|=|f_0|+|f_1|+|f_2|=\frac{1}{2}((|f_0|+|f_1|)+(|f_0|+|f_2|)+(|f_1|+|f_2|))\geq 2^{3(n-j)-i}\cdot3^{i+j-n}.$$

Lemma \ref{constructions of eigenfunctions} implies that if $f\in{F_{3}(n,i,j)}$, then $f\in{U_{[i,j]}(n,3)}$ and $|f|=2^{3(n-j)-i}\cdot 3^{i+j-n}$.
Thus, the bound (\ref{Bound 4}) is sharp.
\end{proof}
Using Theorem \ref{Theorem 5} for $i=j$, we immediately obtain the following result.
\begin{corollary}
  Let $f\in{U_{i}(n,3)}$, where $\frac{n}{2}<i\leq \frac{2n}{3}$ and $f\not\equiv 0$. Then
$$|f|\geq 2^{3n-4i}\cdot3^{2i-n}$$
 and this bound is sharp.
\end{corollary}

\section{Problem \ref{ProblemHamming} for $q=3$ and $\frac{i}{2}+j>n$}\label{SectionProblem2 for q=3,i/2+j>n}
In this section we consider Problem \ref{ProblemHamming} for $q=3$ and $\frac{i}{2}+j>n$. The main result of this section is the following.
\begin{theorem}\label{Theorem 6}
Let $f\in{U_{[i,j]}(n,3)}$, where $\frac{i}{2}+j> n$ and $f\not\equiv 0$. Then
\begin{equation}\label{Bound 5}
|f|\geq 2^{i+j-n}\cdot3^{n-j}
\end{equation}
 and this bound is sharp.
\end{theorem}
\begin{proof}
Let us prove the bound (\ref{Bound 5}) by induction on $n$, $i$ and $j$.
If $n=1$ and $\frac{i}{2}+j>n$, then $i=j=1$. In this case $f\in U_{1}(1,3)$ and the inequality $|f|\geq 2$ holds.

Let us prove the induction step for $n\geq 2$.
Let us consider the functions $f_{0}^{n}$, $f_{1}^{n}$ and $f_{2}^{n}$.
Denote $f_k=f_{k}^{n}$ for $k\in \Sigma_3$.
Let $S=\{s\in \Sigma_3~|~f_s\not\equiv 0\}$.
Let us consider three cases depending on $|S|$.

In the first case we suppose $|S|=1$.
Without loss of generality, we assume that $S=\{0\}$.
Thus $f_1\equiv 0$ and $f_2\equiv 0$.
Then Lemma \ref{Reduction 2} implies that $f_0\in{U_{[i,j-1]}(n-1,3)}$.
We note that $\frac{i}{2}+j-1> n-1$.
Applying the induction assumption for $f_0$, we obtain that $$|f_0|\geq 2^{i+j-n}\cdot3^{n-j}.$$
Then we have $$|f|=|f_0|\geq 2^{i+j-n}\cdot3^{n-j}.$$

In the second case we suppose $|S|=2$.
Without loss of generality, we assume that $S=\{0,1\}$.
So, $f_2\equiv 0$ and $f_k\not\equiv 0$ for any $k\in \{0,1\}$.
Lemma \ref{Reduction 1} implies that $f_k-f_2\in{U_{[i-1,j-1]}(n-1,3)}$ for any $k\in \{0,1\}$.
Consequently, $f_k\in{U_{[i-1,j-1]}(n-1,3)}$ for any $k\in \{0,1\}$.
Applying the induction assumption for $i+2j>2n+1$ and Theorem \ref{Theorem 5} for $i+2j=2n+1$, we obtain that
$$|f_k|\geq 2^{i+j-n-1}\cdot3^{n-j}$$ for any $k\in \{0,1\}$.
Then we have $$|f|=|f_0|+|f_1|\geq  2^{i+j-n}\cdot3^{n-j}.$$

In the third case we suppose $|S|=3$.
So $f_k\not\equiv 0$ for any $k\in \Sigma_3$.
Lemma \ref{Reduction 1} implies that $f_k\in{U_{[i-1,j]}(n-1,3)}$ for $j<n$ and any $k\in \Sigma_3$ and $f_k\in{U_{[i-1,n-1]}(n-1,3)}$ for $j=n$ and any $k\in \Sigma_3$.
We note that $\frac{i-1}{2}+j> n-1$ for $j<n$.
Applying the induction assumption, we obtain that
$$|f_k|\geq 2^{i+j-n}\cdot3^{n-j-1}$$
for $j<n$ and any $k\in \Sigma_3$ and
$|f_k|\geq 2^{i-1}$
for $j=n$ and any $k\in \Sigma_3$.
Then we have $$|f|=|f_0|+|f_1|+|f_2|\geq  2^{i+j-n}\cdot3^{n-j}$$ for $j<n$ and
 $$|f|=|f_0|+|f_1|+|f_2|\geq  3\cdot2^{i-1}>2^i$$ for $j=n$.

Lemma \ref{constructions of eigenfunctions} implies that if $f\in{F_{4}(n,i,j)}$, then $f\in{U_{[i,j]}(n,3)}$ and $|f|=2^{i+j-n}\cdot 3^{n-j}$.
Thus, the bound (\ref{Bound 5}) is sharp.
\end{proof}

Using Theorem \ref{Theorem 6} for $i=j$, we immediately obtain the following result.
\begin{corollary}\label{Corollary2}
Let $f\in{U_{i}(n,3)}$, where $i>\frac{2}{3}n$ and $f\not\equiv 0$. Then
$$|f|\geq 2^{2i-n}\cdot3^{n-i}$$
and this bound is sharp.
\end{corollary}

\section{1-perfect bitrades in the Hamming graph $H(n,3)$}\label{SectionBitrades}
In this section we study 1-perfect bitrades in $H(n,3)$. Firstly, we prove the following result.
\begin{lemma}\label{f(T0,T1)}
Let $(T_0,T_1)$ be a $1$-perfect bitrade in a graph $G$. Then $f_{(T_0,T_1)}$ is a $(-1)$-eigenfunction of $G$.
\end{lemma}
\begin{proof}
Let $x$ be a vertex of $G$. By the definition of a $1$-perfect bitrade we obtain that $$\sum_{y\in B(x)}f_{(T_0,T_1)}(y)=0.$$
Therefore $$f_{(T_0,T_1)}(x)=-\sum_{y\in N(x)}f_{(T_0,T_1)}(y),$$ i.e.
$f_{(T_0,T_1)}$ is a $(-1)$-eigenfunction of $G$.
\end{proof}

\begin{corollary}\label{n=qm+1}
Let $q=p^k$, where $p$ is a prime and $k\geq 1$.
Then $H(n,q)$ has a $1$-perfect bitrade if and only if $n=qm+1$ for some $m\geq 1$.
\end{corollary}
\begin{proof}
Firstly, we note that there are several constructions of $1$-perfect bitrades in $H(qm+1,q)$ (for example, see Lemma \ref{BitradeConstruction}).

Suppose that $(T_0, T_1)$ is a $1$-perfect bitrade in $H(n,q)$.
By Lemma \ref{f(T0,T1)} we obtain that $f_{(T_0,T_1)}$ is a $(-1)$-eigenfunction of $H(n,q)$.
Then $-1=\lambda_{i}(n,q)$ for some $i$. Hence $n=qm+1$ for some $m\geq 1$.
\end{proof}
Thus, we can consider Problem 3 only for $n=qm+1$, where $m\geq 1$.
Now we prove the main result of this section.
\begin{theorem}\label{TheoremBitrade}
The minimum size of a $1$-perfect bitrade in $H(3m+1,3)$, where $m\ge 1$, is $2^{m+1}\cdot 3^{m}$.
\end{theorem}
\begin{proof}
Let $(T_0,T_1)$ be a $1$-perfect bitrade in $H(3m+1,3)$.
Firstly, let us prove the bound
\begin{equation}\label{BoundBitrade}
|T_0|+|T_1|\geq 2^{m+1}\cdot 3^m
\end{equation}
Lemma \ref{f(T0,T1)} implies that $f_{(T_0,T_1)}$ is a $(-1)$-eigenfunction of $H(3m+1,3)$.
We note that $-1=\lambda_{2m+1}(3m+1,3)$.
Applying Corollary \ref{Corollary2} for $n=3m+1$ and $i=2m+1$ ($i>\frac{2}{3}n$), we obtain that
$$|f_{(T_0,T_1)}|\geq 2^{m+1}\cdot 3^{m}.$$
Then we have $$|T_0|+|T_1|=|f_{(T_0,T_1)}|\geq 2^{m+1}\cdot 3^{m}.$$

So, it remains to prove that the bound (\ref{BoundBitrade}) is sharp. Using Lemma \ref{BitradeConstruction} for $q=3$, we see that the bound (\ref{BoundBitrade}) is sharp.
\end{proof}

\begin{remark}
We note that for $m=1$ the claim of Theorem \ref{TheoremBitrade} was recently proved by Mogilnykh and Solov'eva in \cite{MogilnykhSolov'eva}.
\end{remark}

\begin{remark}
Let $f\in F_4(3m+1,2m+1,2m+1)$.
Denote $T_0=\{x\in \Sigma_{3}^{3m+1}~|~f(x)=c\}$ and $T_1=\{x\in \Sigma_{3}^{3m+1}~|~f(x)=-c\}$.
It is easy to verify that $(T_0,T_1)$ is a $1$-perfect bitrade in $H(3m+1,3)$.
Moreover, this $1$-perfect bitrade coincides with the $1$-perfect bitrades (for $q=3$) constructed by Krotov and Vorob'ev in \cite{VorobevKrotov} and Mogilnykh and Solov'eva in \cite{MogilnykhSolov'eva}.
\end{remark}

\section{Acknowledgements}
The author is grateful to Denis Krotov, Ivan Mogilnykh and Konstantin Vorob'ev for helpful and stimulating discussions.

\end{document}